\documentclass{amsart}
\usepackage[latin9]{inputenc}
\usepackage{color}
\usepackage{amsthm}
\usepackage{amssymb}
\usepackage[unicode=true,
 bookmarks=false,
 breaklinks=false,pdfborder={0 0 1},backref=section,colorlinks=true]
 {hyperref}
\hypersetup{
 linkcolor=blue,filecolor=magenta,urlcolor=cyan}

\makeatletter
\numberwithin{equation}{section}
\numberwithin{figure}{section}
\theoremstyle{plain}
\newtheorem{thm}{\protect\theoremname}
\theoremstyle{definition}
\newtheorem{defn}[thm]{\protect\definitionname}
\theoremstyle{definition}
\newtheorem{example}[thm]{\protect\examplename}

\usepackage{graphics}
\usepackage{dsfont}

\usepackage{pdfpages}

\makeatletter
\@namedef{subjclassname@2020}{\textup{2020} Mathematics Subject Classification}
\makeatother

\newtheorem{theorem}{Theorem}[section]

\theoremstyle{definition}

\theoremstyle{corollary}
\newtheorem{corollary}[theorem]{Corollary}

\theoremstyle{note}

\theoremstyle{remark}

\numberwithin{equation}{section}

\newcommand{\bs}{\backslash}

\makeatletter
\@namedef{subjclassname@2010}{%
  \textup{2010} Mathematics Subject Classification}
\makeatother

\makeatother

\providecommand{\definitionname}{Definition}
\providecommand{\examplename}{Example}
\providecommand{\theoremname}{Theorem}

\begin{document}

\title{Studies of certain classes of functions and its connection with $S$-embeddedness }

\author{Biswajit Mitra }

\address{Department of Mathematics. The University of Burdwan, Burdwan 713104,
West Bengal, India }

\email{bmitra@math.buruniv.ac.in}

\author{Sanjib Das }

\address{Department of Mathematics. The University of Burdwan, Burdwan 713104,
West Bengal, India}

\email{ruin.sanjibdas893@gmail.com}



\thanks{The second author's research grant is supported by CSIR, Human Resource
Development group, New Delhi-110012, India }

\begin{abstract}
We call a function $f$ in $C(X)$ to be hard-bounded if $f$ is bounded on every hard subset, a special kind of closed subset, of $X$. We call a subset $T$ of $X$ to be $S$-embedded if every hard-bounded continuous function of $T$ can be continuously extended upto $X$. Every $S$-embedded subset is $C^*$-embedded. In this paper we have given a characterization of the converse part. To get the converse, we came across a type of function which are bounded away from zero on every hard subset of a subset. We further studied few properties of this type of functions and also of hard-bounded functions.
\end{abstract}
\maketitle

\section{introduction}

Through out this paper, we shall always assume a space to be Tychonoff
unless otherwise mentioned. For a space $X$, $C(X)$ and $C^{*}(X)$
usually denote respectively the rings of real-valued and bounded real-valued
continuous functions on $X$. $\beta X$ and $\upsilon X$ respectively
denote the Stone-$\check{C}$ech compactification and Hewitt realcompactification
of a space $X$. A subset $H$ of $X$ is called hard in $X$ if $H$
is closed in $X\cup K$ where $K=cl_{\beta X}(\upsilon X\backslash X)$.
The notion of hard set was introduced by M. C. Rayburn in the year
1976 \cite{gj60}. In our last paper \cite{mdC},  ``Construction of Nearly
Pseudocompactifications''  we introduced the notion of $S$-embeddedness.
A continuous function $f$ on $X$ is called hard-bounded if $f$
is bounded on every hard subset of $X$. Let $S(X)$ denote the family
of all hard-bounded continuous functions. Then $S(X)$ forms a subring
of $C(X)$ containing $C^{*}(X)$. A subset $T$ of $X$ is called
$S$-embedded in $X$ if every hard-bounded continuous function of
$T$ has a continuous extension upto $X$. In our last paper, we introduced
this notion in order to study some extension properties of nearly
pseudocompactification. No further investigations were done in that
paper.

In this paper, we are focused in detailing of $S$-embeddedness. We
first proved a criteria for $S$-embeddedness. We have shown, a $C^{*}$-embedded
subset $T$ of $X$ is $S$-embedded if and only if $T$ is completely
separated from every zero set $Z(g)$ in $X$ where $g$ is bounded
away from zero on every hard subset of $T$. Then we have obtained
some necessary conditions of a continuous function $f$ which is bounded
away from zero on every hard subset of $T$. Indeed we have shown
that within the class of pseudocompact spaces a real-valued continuous
function is bounded away from zero on every hard subset of $T$ if
and only if $Z(f)$ and any hard subset of $T$ are completely separated.
Finally we have investigated the following problem. If a sequence
of hard-bounded continuous functions converges uniformly to a function
$f$ then under what restriction, is $f$ hard-bounded ? We have given
a necessary and sufficient condition to the above problem.

\section{Preliminaries}

In this paper, we used the most of preliminary concepts, notations
and terminologies from the classic monograph of L.Gillman and M.Jerison,
Rings of Continuous Functions \cite{gj60}. However for ready references,
we recall few notations, frequently used over here. For any $f\in C(X)$
or $C^{*}(X)$, $Z(f)=\{x\in X:f(x)=0\}$, called zero set of $f$.
Complement of zero set is called cozero set or cozero part of $f$,
denoted as $cozf$. For any $f\in C(X)$ or $C^{*}(X)$, $cl_{X}(X\bs Z(f))$
is called the support of $f$. Two subsets $A$ and $B$ of $X$ are
said to be completely separated in $X$ if $A$ and $B$ are contained
in two disjoint zero sets. A subspace $Y$ of $X$ is called $C$-embedded
in $X$ if every function in $C(Y)$ can be extended to a function
in $C(X)$ and $Y$ is $C^{*}$-embedded in $X$ if every function
in $C^{*}(Y)$ can be extended to a continuous function in $C(X)$.
If $f\in C(X)$ is unbounded on $E$, then $E$ contains a copy of
$\mathds{N}$, $C$-embedded in $X$ along which $f$ tends to $\infty$.
A function $f$ in $C(X)$ is called bounded away from zero on $X$
if there exists $m>0$ such that $|f(x)|\geq m$ for all $x\in X$.
A non-zero function $f\in C(X)$ (or $C^{*}(X)$) is a unit if and
only if $Z(f)=\phi$(or $f$ is bounded away from zero). A space is
realcompact if and only if every $z$-ultrafilter with countable intersection
property is fixed. In the year 1976, Rayburn \cite{r76} introduced
hard set.
\begin{defn}
A subspace $H$ of $X$ is called hard in $X$ if $H$ is closed in
$X\cup K$, $K=cl_{\beta X}(\upsilon X\bs X)$ where $\beta X$ and
$\upsilon X$ are the Stone-$\check{C}$ech compactification and Hewitt
realcompactification of $X$ respectively.
\end{defn}

It immediately follows that every hard set is closed in $X$, but
the converse is obviously not true. Clearly every compact subset of
$X$ is hard, but the converse may not be true. A hard set is compact
if and only if $X$ is nearly pseudocompact.

\section{main results}

In this section we shall discuss about hard-bounded continuous functions
and $S$-embedded subsets of a space $X$.
\begin{defn}
A continuous function is said to be hard-bounded if it is bounded
on every hard subsets of $X$.
\end{defn}

Let $X$ be a Tychonoff space and $S(X)$ denotes the family of all
hard-bounded continuous functions on $X$, that is $S(X)=\{f\in C(X):f$
is bounded on every hard set in $X\}$. Then $S(X)$ is an intermediate
subring of $C(X)$ containing $C^{*}(X)$. It is given in \cite{hr80} that $X$ is 
nearly pseudocompact if and only if $S(X) = C(X)$. It will be an interesting project 
to investigate those spaces where $S(X)=C^*(X)$. However our present interest lies on $S$-embeddedness and some related
results. 
\begin{defn}
A subset $T$ of $X$ is said to be $S$-embedded if any hard bounded
continuous function on $T$ can be continuously extended upto $X$.

It follows that every $S$-embedded subset is $C^{*}$-embedded. We
now give a necessary and sufficient condition for a $C^{*}$-embedded
subset to be $S$-embedded.
\end{defn}

\begin{thm}
A $C^{*}$-embedded subset $T$ of $X$ is $S$-embedded if and only
if $T$ is completely separated from every zero set $Z(g)$ in $X$
where $g$ is bounded away from zero on every hard subset of $T$.
\end{thm}

\begin{proof}
Suppose, $T$ be a $C^{*}$-embedded subset of $X$, which is $S$-embedded
in $X$, that is every hard-bounded continuous function on $T$ can
be continuously extended upto $X$. Now let $Z(g)$ be a zero set
in $X$ such that $g$ is bounded away from zero on every hard subset
of $T$. Clearly $Z(g)\cap T=\phi$ and $g$ is bounded away from
zero on every hard zero subset of $T$, so $\frac{1}{g}|_{T}=\gamma$
is a hard-bounded cotinuous function on $T$. Let $H$ be a hard set
in $T$, then there exists $m>0$ such that $|g(x)|\geq m$, for all
$x\in H$ $\Rightarrow$ $|\frac{1}{g}(x)|\leq\frac{1}{m}$, for all
$x\in H$ $\Rightarrow$$|\gamma(x)|\leq\frac{1}{m}$, for all $x\in H$
$\Rightarrow$$\gamma$ is bounded on every hard set in $T$. So $\gamma$
is a hard-bounded continuous function on $T$. Since $T$ is $S$-embedded
in $X$, so there exists $h\in C(X)$ such that $h|_{T}=\gamma$.
Hence $hg\in C(X)$. Now on $T$, $hg(x)=h(x).g(x)=\gamma(x).g(x)=\frac{1}{g(x)}.g(x)=1$
and $hg(x)=h(x).g(x)=0$ when $x\in Z(g)$. Therefore $T$ and $Z(g)$
are completely separated.

Conversely, let $T$ is completely separated from every zero set $Z(g)$
in $X$ where $g$ is bounded away from zero on every hard subset
of $T$. Let $T$ is $C^{*}$-embedded in $X$. Let $f\in S(T)$,
that is $f$ is hard-bounded continuous function on $T$. To show
there exists $u\in C(X)$ such that $u|_{T}=f$. Take $tan^{-1}f\in C^{*}(T)$.
As $T$ is $C^{*}$-embedded in $X$, there exists a continuous map
$h:X\to\mathds{R}$ such that $h|_{T}=tan^{-1}f$. Now let $Z=\{x\in X:|h(x)|\geq\frac{\pi}{2}\}=\{x\in X:(|h(x)|-\frac{\pi}{2})\geq0\}=Z((|h(x)|-\frac{\pi}{2})\wedge0)=Z(g)$
where $g=(|h|-\frac{\pi}{2})\wedge0$. Clearly $Z(g)\cap T=\phi$.
Remaining to show that $g$ is bounded away from zero on every hard
set in $T$. Let $H$ be a hard set in $T$. Then there exists $m>0$
such that $|f(x)|\leq m$, for all $x\in H$, as $f\in S(T)$. Now
$tan^{-1}|f(x)|<tan^{-1}m<tan^{-1}(m+1)<\frac{\pi}{2}$, since $tan^{-1}$
is an increasing map. Again $-m<f(x)<m\Rightarrow-\frac{\pi}{2}<tan^{-1}(-m)<tan^{-1}f(x)<tan^{-1}m<\frac{\pi}{2}\Rightarrow-\frac{\pi}{2}<-tan^{-1}m<tan^{-1}f(x)<tan^{-1}m<\frac{\pi}{2}$.
Therefore we have $|tan^{-1}f(x)|<|tan^{-1}m|\Rightarrow|tan^{-1}f(x)|-\frac{\pi}{2}<|tan^{-1}m|-\frac{\pi}{2}<0\Rightarrow(|tan^{-1}f(x)|-\frac{\pi}{2})\wedge0<(|tan^{-1}m|-\frac{\pi}{2})\wedge0=|tan^{-1}m|-\frac{\pi}{2}$,
as $|tan^{-1}m|-\frac{\pi}{2}<0$. Hence $|(|tan^{-1}f(x)|-\frac{\pi}{2})\wedge0|>|(|tan^{-1}m|-\frac{\pi}{2})|>0$,
that is $|(|h(x)|-\frac{\pi}{2})\wedge0|>|(|tan^{-1}m|-\frac{\pi}{2})|$
or, $|g(x)|>u_{m}$, for all $x\in H$ where $u_{m}=|(|tan^{-1}m|-\frac{\pi}{2})|>0\Rightarrow$$g$
is bounded away from zero on every hard set in $T$. So we can completely
separate $T$ and $Z(g)$. There exists $l\in C(X)$ such that $l(T)=\{1\}$
and $l(Z(g))=\{0\}$ and $|l|\leq1$, then $lh\in C(X)$ with $|lh|<\frac{\pi}{2}$.
So $tan\circ lh\in C(X)$ and for $x\in T$, $(tan\circ lh)(x)=tan(l(x).h(x))=tan(1.h(x))=tan(h(x))=tan(tan^{-1}f(x))=f(x)$,
that is $(tan\circ lh)|_{T}=f$. Therefore $T$ is $S$-embedded in
$X$.
\end{proof}
\begin{thm}
$f$is hard-bounded continuous function on $X$ if and only if f is
bounded on every realcompact cozero set in $X.$
\end{thm}

\begin{proof}
Let $f$ be a hard-bounded continuous function on $X.$ Let $P$ be
a realcompact cozero set in $X.$ If $f$ is unbounded on $P.$Then
$P$ contains a $C-embedded$ copy of $\mathbb{\mathbb{N}}$ along
which $f$ tends to infinity. Then $\mathbb{N}$and $X\backslash P$
are completely separated. So $\mathbb{N}$ is hard in $X$ on which
$f$ is unbounded, a contradiction.
\end{proof}
Conversely, let $H$ be a hard set in $X.$ There exists a compact
set $K$ which satisfies the property of hardness. There exist a positive
real number $r$ such that $K\subseteq{x\in X:|f(x)|<r}=V(say)$.
Then $H\backslash V$ is contained in a realcompact cozero set. Thus
$f$ is bounded on $H\backslash V$ also. Hence $f$ is bounded on
$H.$
\begin{thm}
If $Y$ is expressed as finite union of hard sets and $Y$ is $C^{*}$-embedded
then $Y$ is $S$-embedded.
\end{thm}

\begin{proof}
Let $Y=\bigcup_{i=1}^{m}H_{i}$. Let $g\in C(X)$ such that $g$ is
bounded away from zero on every hard subset of $Y$. For each $i$,
there exists $n_{i}>0$ such that $Z_{n_{i}}^{g}=\{x:|g(x)|\geq n_{i}\}\supset H_{i}\Rightarrow\bigcup_{i=1}^{m}Z_{n_{i}}^{g}$
is a zero set and call it $Z$. Then $Z\cap Z(g)=\phi$ and $Y\subset Z.$
So $Z(g)$ is completely separated with $Y$. Hence $Y$ is $S$-embedded.
\end{proof}
We shall now investigate nature of hard-bounded continuous functions
under uniform convergence.
\begin{thm}
Let $\{f_{n}\}$ be a sequence of hard-bounded continuous function
such that $f_{n}\rightarrow f$ uniformly. Then $f$ is hard-bounded
continuous if and only if for every hard set $H$ in $X$, there exists
a sequence $\{m_{n}^{H}\}$ and $m^{H}$ such that $m_{n}^{H}\leq m^{H}$
for all $n\geq k$ for some integer $k$ and $|f_{n}(H)|\leq m_{n}^{H}$,
$|f(H)|\leq m_{H}$.
\end{thm}

\begin{proof}
Suppose $f$ is hard-bounded continuous function on $X$. Let $H$
be a hard set in $X$. Let $m_{n}^{H}=sup_{x\in H}|f_{n}(x)|$. Now
for all $x\in H$, $|f_{n}(x)|\leq|f_{n}(x)-f(x)|+|f(x)|$. As $f_{n}$
converges to $f$ uniformly, for any $\epsilon>0$, there exists $n>0$
such that for all $m\geq n$, $|f_{m}(x)-f(x)|<\epsilon$. So $|f_{m}(x)|\leq|f_{m}(x)-f(x)|+|f(x)|<\epsilon+|f(x)|$
for all $x\in H$. Since $f$ is hard-bounded, there exists $m^{H}$
such that $|f(x)|\leq m^{H}$ for all $x\in H$ $\Rightarrow$$|f_{m}(x)|<\epsilon+m^{H}$
for all $x\in H$ $\Rightarrow$$sup_{x\in H}|f_{m}(x)|\leq\epsilon+m^{H}\Rightarrow m_{n}^{H}\leq\epsilon+m^{H}$.
Letting $\epsilon\rightarrow0$, we get $m_{m}^{H}\leq m^{H}$ for
all $m\geq n$.

Conversely, let for every hard set $H$ in $X$, there exists a sequence
$\{m_{n}^{H}\}$ and $m^{H}$ such that $m_{n}^{H}\leq m^{H}$ for
all $n\geq k$ for some integer $k$ and $|f_{n}(H)|\leq m_{n}^{H}$,
$|f(H)|\leq m_{H}$. We have to show that $f$ is hard-bounded. Let,
$H$ be a hard set in $X$. Therefore $|f_{n}(H)|\leq m_{n}^{H}$
and $m_{n}^{H}\leq m^{H}$ for all $n\geq k$ $\Rightarrow$$|f_{n}(H)|\leq m^{H}\Rightarrow|f(H)|\leq m^{H}$.
Hence $f$ is hard-bounded continuous map.
\end{proof}
In next few theorems, we discussed about those function which are
bounded away from zero on every hard subset of a subset of $X$:
\begin{thm}
Let $X$ be a pseudocompact space. An $f\in C(X)$ is bounded away
from zero on every hard subset of $T$ if and only if $Z(f)\cap T=\phi$.
\end{thm}

\begin{proof}
Suppose $f\in C(X)$ is bounded away from zero on every hard subset
of $T$. Then, $\{x\}$ being hard subset of $T$, $f(x)\neq0$ for
all $x\in T$. Therefore $Z(f)\cap T=\phi$.

Conversely, suppose $Z(f)\cap T=\phi$. We have to show that, $f$
is bounded away from zero on every hard subset of $T$. Let us assume
that $f$ is not bounded away from zero on every hard subset of $T$.
Then there exists a hard subset $H$ of $T$ in which $f$ is not
bounded away from zero. So for every $m>0$, $H\nsubseteq\{x\in X:|f(x)|\geq m\}\Rightarrow\forall m$
there exists $x_{m}\in H$ such that $|f(x_{m})|<m$. Again, as $Z(f)\cap H=\phi$,
so $|f(x)|>0$, $\forall m$. Then we have a sequence $\{x_{m}\}$
in $H$ satisfying the following properties: $|f(x_{1})|<1$ and $|f(x_{n})|<min\{|f(x_{n-1})|,\frac{1}{n}\}$.
So $|f(x_{n})|<|f(x_{n-1})|$ and $|f(x_{n})|<\frac{1}{n}$, for all
$n\in\mathds{N}$. Now we choose a sequence of closed intervals $\{I_{n}\}$
such that $|f(x_{n})|\in intI_{n}$, for all $n\in\mathds{N}$ and
$I_{n}\cap I_{m}=\phi$, for all $n\neq m$. Accordingly we have a
sequence of closed sets $\{V_{n}\}$ such that $x_{n}\in intV_{n}\subset V_{n}$
and for all $x\in V_{n}$, $|f(x)|\in I_{n}$. Clearly $V_{n}\cap V_{m}=\phi$
for $m\neq n$ and $V_{n}\neq\phi$ for all $n$. We shall now show
that for $m\in\mathds{N}$, $\bigcup_{n\neq m}V_{n}$ is closed in
$X$. Let $x\in X\backslash\bigcup_{n\neq m}V_{n}$. Then $|f(x)|>0$
and there exists a $k$ such that $|f(x_{k+1})|\leq|f(x)|<|f(x_{k})|$
for $k\leq1$ or $|f(x)|\geq|f(x_{1})|$. In any case there exists
a neighbourhood of $x$ which intersects at most two $V_{k}$'s. Thus
$\{V_{n}:n\neq m\}$ is locally finite and hence $\bigcup_{n\neq m}V_{n}$
is closed in $X$. By \cite{gj60} (Exercise 3L.), $N=\{x_{n}:n\in\mathds{N}\}$
is a copy of $\mathds{N}$, $C$-embedded in $X$. Therefore $X$
is not pseudocompact, a contradiction.
\end{proof}
The following counter example asserts that we can not drop the condition
of pseudocompactness of $X$ in the above theorem.
\begin{example}
Take $X=\mathds{N}$ with discrete topology. Take $T=2\mathds{N}$.
As $T$ is realcompact, hard sets of $T$ are precisely the closed
subsets of $T$. Also note that, any subset of a discrete space is
a zero set. We can construct a continuous function $f:X\rightarrow\mathds{R}$
satisfying $Z(f)=X\backslash T$ and $f(x)=\frac{1}{x}$, whenever
$x\in T$. It is easy to observe that $Z(f)\cap T=\phi$, but $f$
is not bounded away from zero on every hard subset of $T$.
\end{example}

\begin{thm}
Let $T$ be a nearly pseudocompact subset of $X$. An $f\in C(X)$
is bounded away from zero on every hard subset of $T$ if and only
if $Z(f)\cap T=\emptyset.$
\end{thm}

\begin{proof}
Following the proof of Theorem 3 above, the set $N$ is hard in $T$.
As $T$ is nearly pseudocompact, $N$ is compact which is absurd.
Rest follows from the above proof of Theorem{[}3{]}.
\end{proof}
\begin{thm}
$f\in S(X)$ is a unit of $S(X)$ if and only if $f$ is bounded away
from zero on every hard subset of $X$.
\end{thm}

\begin{proof}
Let $f\in S(X)$ be a unit of $S(X)$. So there exists $g\in S(X)$
such that $f.g=1$. Let $H$ be a hard subset of $X$. Since $g\in S(X)$,
$|g(x)|\leq m$ for all $x\in H$, for some $m>0$ $\Rightarrow$$|f(x)|\geq\frac{1}{m}$
for all $x\in H$$\Rightarrow$$f$ is bounded away from zero on $H$.
Hence $H$ being arbitrary, $f$ is bounded away from zero on every
hard subset of $X$.

Conversely, suppose that $f\in S(X)$ is bounded away from zero on
every hard subset of $X$. Then $f(x)\neq0$ for all $x\in X$, since
$\{x\}$ being compact, it is hard and hence $f$ can not be zero
at $x$. So $\frac{1}{f}$ exists and it is continuous on $X$. Now
we have to show $\frac{1}{f}\in S(X)$. Let $H$ be a hard subset
of $X$. By our assumption $|f(x)|\geq m_{H}$ for all $x\in H$ for
some $m_{H}>0$ $\Rightarrow$$\frac{\ensuremath{1}}{|f(x)|}\leq\frac{1}{m_{H}}$
for all $x\in H$ i.e. $|\frac{1}{f}(x)|\leq\frac{1}{m_{H}}$ for
all $x\in H$ $\Rightarrow$$\frac{1}{f}$ is bounded on $H$ . Hence
$H$ being arbitrary, $\frac{1}{f}$ is bounded on every hard subset
of $X$ i.e. $\frac{1}{f}\in S(X)$. Therefore $f$ is a unit of $S(X)$.
\end{proof}
\begin{corollary}
A $C^{*}$-embedded subset $T$ of $X$ is $S$-embedded if and only
if for all $f\in C(X)$, $f|_{T}$ is a unit of $S(T)$ and $Z(f)$,
$T$ are completely separated.
\end{corollary}

Acknowledgements: The authors sincerely acknowledge the support received
from DST FIST programme (File No. SR/FST/MSII/2017/10(C))

\end{document}